\documentclass[12pt]{amsart}
\usepackage{fullpage}
\usepackage{amsthm, amstext, amsmath, amssymb,soul}
\usepackage{graphicx}
\usepackage{verbatim}
\usepackage{listings, xcolor}
\usepackage{mathtools}
\usepackage{hyperref}
\usepackage[normalem]{ulem}
\usepackage{soul} 

\title{New Normal Forms For Degree Three Polynomials and Rational Functions}

\author{Heidi Benham}
\address{Department of Mathematics,
         Western Oregon University,
         Monmouth OR 97361}
\email{hbenham17@wou.edu}

\author{Alexander Galarraga}
\address{Department of Mathematics,
         University of Washington,
         Seattle, WA, 98195}
\email{agalar@uw.edu}

\author{Benjamin Hutz}
\address{Department of Mathematics and Statistics,
         Saint Louis University,
         St.~Louis, MO 63103}
\email{benjamin.hutz@slu.edu}

\author{Joey Lupo}
\address{Department of Mathematics and Statistics,
         Amherst College,
         Amherst, MA 01002}
\email{jlupo20@amherst.edu}

\author{Wayne Peng}
\address{Department of Mathematics,
        University of Rochester
        Rochester, NY 14627}
\email{jpeng4@ur.rochester.edu}

\author{Adam Towsley}
\address{School of Mathematical Sciences,
        Rochester Institute of Technology
        Rochester, NY 14623}
\email{adtsma@rit.edu}

\subjclass[2010]{
37P45,   	
37P05,   	
(37P15)     
}

\keywords{dynamical system, normal form, moduli space}

\newtheorem{thm}{Theorem}
\newtheorem{cor}[thm]{Corollary}
\newtheorem{lem}[thm]{Lemma}

\newtheorem*{conj*}{Conjecture}

\theoremstyle{definition}
\newtheorem{defn}[thm]{Definition}
\newtheorem{ex}[thm]{Example}
\newtheorem{rem}[thm]{Remark}

\def\AA{\mathbb{A}}

\def\QQ{\mathbb{Q}}
\def\Qbar{\overline{\QQ}}

\def\CC{\mathbb{C}}
\def\PP{\mathbb{P}}
\def\cP{\mathcal{P}}
\def\cM{\mathcal{M}}

\DeclareMathOperator{\PGL}{PGL}
\DeclareMathOperator{\Aut}{Aut}
\DeclareMathOperator{\Gal}{Gal}
\DeclareMathOperator{\Fix}{Fix}

\providecommand{\open}[1]{\textbf{Open:} \textcolor{red}{#1}}
\providecommand{\new}[1]{\textcolor{blue}{#1}}
\def\wayne#1{\textcolor{magenta}{#1}}

\definecolor{codegreen}{rgb}{0,0.6,0}
\definecolor{codegray}{rgb}{0.5,0.5,0.5}
\definecolor{codepurple}{rgb}{0.58,0,0.82}
\definecolor{backcolour}{rgb}{0.95,0.95,0.92}
\lstdefinestyle{mystyle}{
   backgroundcolor=\color{backcolour},
   commentstyle=\color{codegreen},
   keywordstyle=\color{magenta},
   numberstyle=\tiny\color{codegray},
   stringstyle=\color{codepurple},
   basicstyle=\ttfamily\footnotesize,
   breakatwhitespace=false,
   breaklines=true,
   captionpos=b,
   keepspaces=true,
   numbers=left,
   numbersep=5pt,
   showspaces=false,
   showstringspaces=false,
   showtabs=false,
   tabsize=2
}
\begin{document}

\thanks{Thanks to the Institute for Computational and Experimental Research in Mathematics, where most of this work was completed during the Summer 2019 Research Experience for Undergraduates. Thanks to an anonymous referee who provided many helpful comments and corrections.}

\maketitle

\begin{abstract}
    When studying families in the moduli space of dynamical systems, choosing an appropriate representative function for a conjugacy class can be a delicate task. The most delicate questions surround rationality of the conjugacy class compared to rationality of the defining polynomials of the representation. We give a normal form for degree three polynomials which has the property that the set of fixed points is equal to the set of fixed point multipliers. This normal form is given in terms of moduli space invariants and, hence, has nice rationality properties. We further classify all degree three rational maps which can be conjugated to have a similar relationship between the fixed points and the fixed point multipliers.
\end{abstract}

Let $f(z) \in K(z)$ be a rational function of degree $d\geq2$ defined over a perfect field $K$, considered as an endomorphism of $\PP^1$. For example, $f(z) = z^2+1$ represents the endomorphism of $\PP^1$ given by $(x:y) \mapsto (x^2 + y^2 : y^2)$. Define the $n$\emph{-th iterate} of $f$ recursively as $f^n(z) = f(f^{n-1}(z))$, with $f^0(z) = z$. There is a natural conjugation action on $f$ by $\alpha \in \PGL_2$ given by $f^{\alpha} = \alpha^{-1} \circ f \circ \alpha$. Since the dynamical behavior of $f$ is preserved by this conjugation action, we may consider the set of equivalence classes of degree $d$ rational endomorphisms of $\PP^1$ under $\PGL_2$ conjugation. We denote $f$ for maps and $[f]$ for conjugacy classes represented by the map $f$.  We denote this moduli space as $\cM_d$, and denote by $\cP_d \subset \cM_d$ the moduli space of degree $d$ polynomials \cite[Section 4.4]{Silverman}. 
Recall that $[f] \in \cP_d$ if and only if it has a totally ramified fixed point. For our purposes, we say a family of maps $f_t(z)$ provides a \emph{normal form} if each choice of parameter value determines a different conjugacy class of $\cM_d$.

The choice of normal form can be quite delicate. For degree two rational maps, Milnor proved that $\cM_2 \cong \CC^2$ with the isomorphism defined in terms the multipliers of the fixed points \cite{Milnorquadratics}. Let $\lambda_1$, $\lambda_2$, and $\lambda_3$ be the multipliers of the three fixed points and define the invariants of $\cM_2$
\begin{align*}
    \sigma_1 &= \lambda_1 + \lambda_2 + \lambda_3\\
    \sigma_2 &= \lambda_1\lambda_2 + \lambda_1\lambda_3 + \lambda_2\lambda_3\\
    \sigma_3 &= \lambda_1\lambda_2\lambda_3.
\end{align*}
The isomorphism is then given by
\begin{equation*}
    [f] \mapsto (\sigma_1,\sigma_2).
\end{equation*}
Note that for all maps, $\sigma_3 = \sigma_1 - 2$. In particular, the pair of values $(\sigma_1,\sigma_2)$ form coordinates for the moduli space $\cM_2$ and conjugacy classes $[f]$ where $(\sigma_1,\sigma_2) \in \QQ^2$ are deemed ``rational points'' in the moduli space. The delicate question is: Given $(\sigma_1,\sigma_2) \in \QQ^2$ associated to a conjugacy class $[f]$ is there a representative of $[f]$ which is defined over $\QQ$? Given an appropriate set of coordinates of $\cM_d$, this question is valid in any degree $d>1$ and, more generally, is the ``field of moduli versus field of definition'' problem, see for example \cite[Section 4.10]{Silverman}.

For $\cM_2$, Milnor  \cite{Milnorquadratics} first gave the normal form
\begin{equation} \label{eq_milnor}
    [f] \longleftrightarrow \begin{cases}
      \frac{z^2 + \lambda_1 z}{\lambda_2 z + 1} & \lambda_1\lambda_2 \neq 1\\
      z + \sqrt{1-\lambda_3} + \frac{1}{z} & \lambda_1\lambda_2=1.
    \end{cases}
\end{equation}
However, even if $(\sigma_1,\sigma_2) \in \QQ^2$ it may be that the multipliers themselves, $\lambda_1,\lambda_2$ are defined over a quadratic extension. So this normal form does not have ``nice'' rationality properties. Manes-Yasafuku \cite{Yasufuku} give a normal form (for all but finitely many conjugacy classes),
\begin{equation*}
    [f] \longleftrightarrow \frac{2z^2 + (2-\sigma_1)z + (2-\sigma_1)}{-z^2 + (2+\sigma_1)z + 2-\sigma_1-\sigma_2}.
\end{equation*}
Since the normal form is defined in terms of $(\sigma_1,\sigma_2)$, it has nice rationality properties. Interestingly, this normal form satisfies that the fixed points are equal to their multipliers. Theorem \ref{thm:2} proves when an analogous normal form exists for degree 3 maps. The precise definition of ``partial fixed-point multiplier form'' is given in Definition \ref{defn_fixed_form}.
\begin{thm}\label{thm:2}
     Let $f$ be a degree 3 rational function defined over a perfect field $K$. Then there exists a $\rho \in \PGL_2(\overline{K})$ such that $f^{\rho}$ is in partial fixed-point multiplier form and defined over $K$ if and only if the following conditions are satisfied
    \begin{enumerate}
        \item There are at least $\min(\#\Fix(f),3)$ distinct fixed point multipliers. \label{cond_1}
        \item One of the following \label{cond_2}
        \begin{enumerate}
            \item $f$ has a $\Gal(\overline{K}/K)$ invariant set of $\min(\#\Fix(f),3)$ distinct fixed points. \label{cond_2a}
            \item The automorphism group of $f$ has order $2$. Moreover, let $\alpha$ be the nontrivial automorphism. Then for any fixed point $x$ and for any $\sigma\in\Gal(\overline{K}/K)$, we have either $\sigma(x)=x$ or $\sigma(x)=\alpha(x)$. \label{cond_2b}
            \item $f^{\rho}$ is in fixed-point multiplier form and is one of the maps in Lemma \ref{lem_4_distinct}.\label{cond_2c}
        \end{enumerate}
    \end{enumerate}
\end{thm}

A main application of these types of normal forms is enumerating all maps satisfying certain properties, such as being post-critically finite; which we now define. A \emph{critical point of $f$} is a point with ramification index at least two. When the forward orbits of all the critical points are finite, we say the map is \emph{post-critically finite (PCF)}.

Ingram proves that the set of PCF maps in $\cP_d$ is a set of bounded height and, hence, finite for any number field \cite{Ingram-height-bound}. He goes on to calculate a specific bound for the coefficients of a degree three PCF polynomial in monic centered form and enumerates all possibilities over $\QQ$. However, in choosing to use monic centered form,
\begin{equation*}
    z^d + a_{n-2}z^{d-2} + \cdots + a_1z + a_0,
\end{equation*}
as his normal form, he did not find all PCF polynomials defined over $\QQ$. The problem was one of the delicate rationality issues: not every PCF degree three polynomial defined over $\QQ$ is conjugate to a polynomial in monic centered form also with coefficients in $\QQ$. Tobin \cite{Bella} improves on his classification with her normal form for bicritical polynomials, but also omits a case due to the same type of rationality issue with conjugation changing the field of definition. Anderson-Manes-Tobin cover this last case, completing the enumeration of all degree three PCF polynomials defined over $\QQ$ \cite{AMT}. The normal form provided in Theorem \ref{normalform} is given in terms of moduli space invariants and solves these subtle rationality type issues.
    \begin{thm}\label{normalform}
        Define a map $\phi: \AA^2 \to \cP_3$, from the affine plane to the moduli space of degree three polynomials as follows

        \begin{equation*}
        \phi(\sigma_1, \sigma_3) =
            \begin{cases}
              \left[\frac{(12-2\sigma_1)z^3 + (2\sigma_1^2 - 15\sigma_1 + 18)z^2 + 2\sigma_1\sigma_3 - 3\sigma_3}{-3z^3 + (9 + 3\sigma_1)z^2 - (18\sigma_1 - 27)z + 4\sigma_1^2 -12\sigma_1 + 3\sigma_3 + 9}\right] & (\sigma_1, \sigma_3) \notin \mathcal{C} \\
              \\
              \left[\frac{(9\sigma_1 - 27)z^3}{(-5\sigma_1+12)z^3 + (\sigma_1^2 +15\sigma_1 - 45)z^2 + (-9\sigma_3)z - \sigma_1\sigma_3 + 6\sigma_3}\right] & (\sigma_1,\sigma_3) \in \mathcal{C},\\
              &\sigma_1 \neq 3,6,\frac{3}{2} \\
              \\
              [z^3] & (\sigma_1, \sigma_3) = (6, 0) \\
              [z^3 + z] & (\sigma_1, \sigma_3) = (3, 1)\\
              \left[z^3 + \frac{3}{2}z\right] &  (\sigma_1, \sigma_3) = (\frac{3}{2}, 0),
           \end{cases}
        \end{equation*}
        where $\mathcal{C}$ is the curve $4\sigma_1^3 - 36\sigma_1^2 + 81\sigma_1 + 27\sigma_3 - 54 = 0$. The map $\phi$ is a bijection, with inverse given by taking the sigma invariants. Further, the first two cases are in partial fixed-point multiplier form.
    \end{thm}
While the above rational functions are polynomials in the sense of having a totally ramified fixed point, they are not in the standard form of polynomials. We perform the necessary conjugation to standard polynomial form to find the following bijection.
\begin{cor}\label{polynomial_normal_form}
Define a map $\phi: \AA^2 \to \cP_3$, from the affine plane to the moduli space of degree three polynomials as follows
\begin{align*}
    &\varphi(\sigma_1, \sigma_3) = \\
    &\begin{cases}
      \Big[\frac{1}{{(16\sigma_1^4 - 144\sigma_1^3 + 468\sigma_1^2 - 648\sigma_1 + 324)}} \Big((4\sigma_1^3 - 36\sigma_1^2 + 81\sigma_1 + 27\sigma_3 - 54)z^3\\
      \qquad+ (8\sigma_1^4 - 96\sigma_1^3 + 378\sigma_1^2 + 54\sigma_1\sigma_3 - 594\sigma_1 - 162\sigma_3 + 324)z^2 & (\sigma_1, \sigma_3) \not \in \mathcal{C} \\
      \qquad + (36\sigma_1^2\sigma_3 - 216\sigma_1\sigma_3 + 324\sigma_3)z\\
      \qquad + (8\sigma_1^3\sigma_3 - 72\sigma_1^2\sigma_3 + 216\sigma_1\sigma_3 - 216\sigma_3)\Big)\Big]\\ \\
      \left[\frac{(-4\sigma_1\sigma_3 + 24\sigma_3)z^3 + (6\sigma_1^2 - 45\sigma_1 - 9\sigma_3 + 54)z^2 + (-2\sigma_1^2 + 33\sigma_1 - 45)z - 2\sigma_1 - 6}{9\sigma_1 - 27}\right]  & (\sigma_1, \sigma_3) \in \mathcal{C}  \\ \\
      [z^3] & (\sigma_1, \sigma_3) = (6, 0) \\
      [z^3 + z] & (\sigma_1, \sigma_3) = (3, 1)\\
      \left[z^3 + \frac{3}{2}z\right] &  (\sigma_1, \sigma_3) = (\frac{3}{2}, 0),
    \end{cases}
\end{align*}
\end{cor}

Our normal form is not, however, the first to solve these rationality issues, nor is it the simplest, as a much simpler normal form appears in \cite[Section 4]{points-of-small-height}. Our normal form does have the advantage that it provides an explicit inverse to the well known map resulting from taking the sigma invariants by parametrizing $\cP_3$ in terms of the sigma invariants. Thus, our normal form provides a way to complete the classification of degree three PCF polynomials utilizing the bound from Benedetto-Ingram-Jones-Levy; however, this classification was completed by Anderson-Manes-Tobin \cite{AMT} while our work was underway so we do not repeat their classification calculations here, but do provide our normal form for future applications.

Moving to the more difficult rational maps case, Lukas-Manes-Yap \cite{LMY} find all degree two PCF rational maps defined over $\QQ$ using the normal form from Manes-Yasafuku \cite{Yasufuku}. This allows them to use the bound on the multipliers of a PCF map from Benedetto-Ingram-Jones-Levy \cite{BIJL} to enumerate all possibilities without the problem of omission encountered by Ingram and Tobin when bounding the coefficients. In Theorem \ref{thm:2}, we prove that the idea behind the normal form from Manes-Yasafuku \cite{Yasufuku}, the fixed points being ``equal'' to their multipliers, is only possible for certain classes of degree three rational maps and so cannot be used for further PCF classification results.

The article is organized as follows. In Section \ref{sect_deg_3} we prove Theorem \ref{normalform}. In Section \ref{sect_fixed_point_form} we investigate the existence of a normal form for degree three rational maps proving Theorem \ref{thm:2} and provide some examples, while in Section \ref{first_dynatomic} we investigate some auxiliary results and corollaries.

\section{Normal Form for Degree 3 Polynomials} \label{sect_deg_3}

    Let $f$ be a single variable rational function. Define a \emph{fixed point} of $f$ to be a point $P\in\Qbar$ such that $f(P) = P$, and the \emph{multiplier at P} to be $\lambda_P := f'(P)$. We may make a linear change of variables to calculate $\lambda_{\infty}$ as needed. Define the \emph{$\sigma$-invariants} (or \emph{sigma invariants}) of $f$ to be the elementary symmetric polynomials evaluated on the set of fixed point multipliers (with multiplicity):
    \begin{equation*}
        \prod_{P \in \Fix(f)}(t - \lambda_P) = \sum_{i=0}^d (-1)^i \sigma_it^{d-i}.
    \end{equation*}
    Because the set of fixed point multipliers is preserved under conjugation, these $\sigma$-invariants are the same for every element of a conjugacy class, i.e., are invariants of $\cM_d$.

   Milnor's \cite{Milnorquadratics} 
   normal form for degree two rational maps is given in terms of the multipliers themselves (Equation \eqref{eq_milnor}). However, such a form may not utilize the smallest possible field of definition since the multipliers could be defined over an extension field. The $\sigma$-invariants are always defined over the field of moduli so a normal form based on the $\sigma$-invariants has the smallest possible field of definition. See Silverman \cite[Section 4.10]{Silverman} for precise definitions for the field of moduli and field of definition and some of the associated subtleties.

    \begin{proof}[Proof of Theorem \ref{normalform}]
        We note that the curve $\mathcal{C}$ corresponds to the resultant of the numerator and denominator of the first form of the image of $\phi$. In other words, if $(\sigma_1, \sigma_3)\in\mathcal{C}$, then there is a common root between the numerator and denominator so that the degree decreases and the map is no longer an element of $\cP_3$. Similarly, computing the resultant of the second form shows that the degree of this form decreases when $\sigma_1\in \{3,6\}$ or $\sigma_3 = 0$. When $\sigma_3 =0$, then $\sigma_1$ is either $6$ or $\frac{3}{2}$, so the three conditions in the second case are required. These remaining three cases for the image of $\phi$ cover all of $\mathcal{P}_3$.

        First, we verify that $\phi$ actually maps into $\mathcal{P}_3$. We can compute in Sage \cite{sage} using the \texttt{DynamicalSystem} functionality that the first form has a totally ramified fixed point at $z=\frac{2}{3}\sigma_1 - 1$ and the second form at $z=0$. The final three cases are clearly polynomials.

        Now let $(\sigma_1,\sigma_3) \in \mathbb{A}^2$. We next show that in each case, the $\sigma$-invariants of $\phi(\sigma_1, \sigma_3)$ are given by $\{\sigma_1, \sigma_2, \sigma_3, 0\}$. First, assume $(\sigma_1, \sigma_3) \notin \mathcal{C}$. From \cite[Proposition 8]{Hutz10}, we can say $\sigma_2 = 2\sigma_1 -3$. We use Sage to compute that the $\sigma$-invariants of $\phi(\sigma_1,\sigma_3)$
        are given by
        \begin{equation*}
            \{\sigma_1, 2\sigma_1 -3, \sigma_3, 0\},
        \end{equation*}
        as desired. Now assume $(\sigma_1, \sigma_3)\in \mathcal{C}$ so that, in particular, we can say $\sigma_3 = \frac{1}{27}(-4\sigma_1^3 + 36\sigma_1^2 - 81\sigma_1 + 54)$. We next compute the $\sigma$-invariants of $\phi(\sigma_1, \sigma_3)$. Since $(\sigma_1, \sigma_3)\in \mathcal{C}$, we take the remainder of each $\sigma_i$ on division by the defining polynomial of $\mathcal{C}$ to obtain
        \begin{equation*}
            \{\sigma_1, 2\sigma_1 -3, \frac{1}{27}(-4\sigma_1^3 + 36\sigma_1^2 - 81\sigma_1 + 54), 0\},
        \end{equation*}
        as desired. The last three cases are easily verified.

        Now it is easy to show that $\phi$ is one-to-one. Assume there is $(\sigma_1', \sigma_3')\in \AA^2$ such that $\phi(\sigma_1, \sigma_3) = \phi(\sigma_1', \sigma_3')$. Combining this assumption with the argument of the preceding paragraph yields
        \begin{equation*}
            \{\sigma_1, \sigma_2, \sigma_3, 0\} = \{\sigma_1', \sigma_2', \sigma_3', 0\},
        \end{equation*}
        so that we conclude $\sigma_1 = \sigma_1'$ and $\sigma_3 = \sigma_3'$. To prove that $\phi$ is onto, we use the fact from Fujimura-Nishizawa \cite{Fujimura} that the set of fixed point multipliers, and thus the set of $\sigma$-invariants, uniquely determines the conjugacy class of a cubic polynomial. In particular, given $[f]\in \cP_3$ with $\sigma$-invariants $\{\sigma_1, \sigma_2, \sigma_3, 0\}$, we conclude that $\phi(\sigma_1, \sigma_3) \in [f]$. It follows that $\phi$ is a bijection.

        Finally, we show that the first two cases are in partial fixed-point multiplier form. For the first case we have $(\sigma_1,\sigma_3) \not\in \mathcal{C}$ and compute the first dynatomic polynomial as
        \begin{equation*}
            (-1/3)(3z - 2\sigma_1 + 3)(z^3 - \sigma_1z^2 + (2\sigma_1 - 3)z - \sigma_3).
        \end{equation*}
        The fixed point $\frac{2\sigma_1-3}{3}$ is totally ramified so has multiplier $0$, and we know $\sigma_2 = 2\sigma_1-3$. Hence, the roots of the degree three factor of the dynatomic polynomial are the three nonzero multiplier values. Since the roots of the first dynatomic polynomial are also the fixed points, we know those fixed points and their multipliers are equal as sets. To establish partial fixed-point multiplier form, we need to see the fixed points correspond to the multiplier with the same value. We check this by specialization to any valid $(\sigma_1,\sigma_3)$ where the fixed points remain distinct.

        For the second case we have $(\sigma_1,\sigma_3) \in \mathcal{C}$ to write
        \begin{equation*}
            \sigma_3 = \frac{1}{27}\left(-4\sigma_1^3 + 36\sigma_1^2 - 81\sigma_1 + 54\right).
        \end{equation*}
        Then the first dynatomic polynomial factors as
        \begin{equation*}
            \frac{1}{27}(3z + (-2\sigma_1 + 3))(3z + (\sigma_1 - 6))((15\sigma_1 - 36)z + (2\sigma_1^2 - 15\sigma_1 + 18))z
        \end{equation*}
        giving the four fixed points
        \begin{equation*}
            \left\{\frac{2}{3}\sigma_1 - 1, 2-\frac{1}{3}\sigma_1, \frac{-2\sigma_1^2 + 15\sigma_1 - 18}{15\sigma_1-36}, 0\right\}.
        \end{equation*}
        The corresponding multipliers can be calculated directly as
        \begin{equation*}
            \left\{\frac{2}{3}\sigma_1 - 1, 2-\frac{1}{3}\sigma_1, \frac{2}{3}\sigma_1 - 1, 0\right\}.
        \end{equation*}
        Thus, this map is in partial fixed-point multiplier form.
    \end{proof}

\section{Fixed-point Multiplier Form}\label{sect_fixed_point_form}


    In the first two cases of Theorem \ref{normalform}, the representative given by $\phi$ has the unique property that the three non-zero fixed point multipliers are three of its fixed points. This property mirrors that of the normal form given in Manes-Yasufuku \cite{Yasufuku}. Recall that there is a unique element of $\PGL_2$ that moves three points in $\PP^1(\CC)$ to three other points. Since the form \cite{Yasufuku} is for $\mathcal{M}_2$, where every map has exactly three fixed points, there is a unique conjugation to a map whose fixed points and fixed point multipliers coincide.
    One might hope there is a more general normal form defined over the field of moduli for $\mathcal{M}_3$ with the property that the fixed-point multipliers equal the fixed points. However, as Theorem \ref{thm:2} clarifies, this is not the case.

    \begin{defn}\label{defn_fixed_form}
        We say that $f$ is in \emph{fixed-point multiplier form} if $f$ has at least three distinct fixed points and all fixed points are equal to their multipliers.
        We say that $f$ is in \emph{partial fixed-point multiplier form} if either
        \begin{enumerate}
            \item $f$ has two or less fixed points and all fixed points are equal to their multipliers.
            \item $f$ has three or more fixed points and at least three of the fixed points are equal to their multipliers.

        \end{enumerate}
    \end{defn}

    \begin{lem} \label{lem_4_distinct}
        Let $K$ be a perfect field. Let $x_1, x_2, x_3, x_4 \in \overline{K} \smallsetminus \{1\}$ form a $\Gal(\overline{K}/K)$ invariant set and satisfy
        \begin{equation*}
            \sum_{i=1}^4 \frac{1}{1-x_i} = 1.
        \end{equation*}
        Then there is a unique degree $3$ rational function $f$ defined over $K$ in fixed-point multiplier form with fixed points $\Fix(f) = \{x_1,x_2,x_3,x_4\}$. Furthermore, every degree 3 rational function defined over $K$ with $4$ distinct fixed points in fixed-point multiplier form can be obtained in this way.
    \end{lem}
    We adapt the proof of Hutz-Tepper \cite[Theorem 6]{Hutz10}, which proves that specifying the four fixed points and their multipliers uniquely determines a degree $3$ rational function. However, the statement in the cited paper only assumes the multipliers are defined as a set, whereas their proof requires knowing which multiplier is attached to which fixed point.
    \begin{proof}
         The values $x_1,x_2,x_3,x_4$ are affine, so the point at infinity is not a fixed point. We dehomogenize and denote  the rational map as $F$. We can write $F$ in the form
         \begin{equation*}
             F(z) = z - \frac{p(z)}{q(z) } = z - \frac{a\prod_{i=1}^4 (z-x_i)}{az^3 + bz^2 + cz + d}.
         \end{equation*}
         Computing
         \begin{equation*}
             F'(z) = 1 - \frac{p'(z)}{q(z)} - \frac{p(z)q'(z)}{q(z)^2}
         \end{equation*}
         and evaluating at the fixed points $\{x_1,x_2,x_3,x_4\}$, we need to have
         \begin{equation*}
             x_i = F'(x_i) = 1 - \frac{p'(x_i)}{q(x_i)}, \quad i=1,2,3,4.
         \end{equation*}
         This gives the system of equations (linear in $a,b,c,d$) defined by
         \begin{equation}\label{eq_1}
             (1-x_i)q(x_i) - p'(x_i) = 0, \quad i=1,2,3,4.
         \end{equation}
         As in Hutz-Tepper, this system has a unique solution for $(a,b,c,d)$ and, hence, for $F$. To see that this solution is defined over $K$ recall that we have assumed $\{x_1,x_2,x_3,x_4\}$ is $\Gal(\overline{K}/K)$ invariant. In particular, applying any element of $\Gal(\overline{K}/K)$ to the system \eqref{eq_1} fixes the system. In other words, elements of $\Gal(\overline{K}/K)$ fix the solution $(a,b,c,d)$. As $K$ is perfect, the fixed field of $\Gal(\overline{K}/K)$ is $K$, and thus the solution $(a,b,c,d)$ and $f$ are defined over $K$.

         For the final statement, assume $f$ is in fixed-point multiplier form defined over $K$ with $4$ distinct fixed points $\Fix(f) = \{x_1,x_2,x_3,x_4\}$. Then the first dynatomic polynomial \cite[Section 4.1]{Silverman} is defined over $K$ so that $\Fix(f)$ is $\Gal(\overline{K}/K)$ invariant. Furthermore, since the fixed points are equal to their multipliers, they must satisfy the classical relation \cite[Theorem 1.14]{Silverman}
         \begin{equation*}
             \sum_{i=1}^4 \frac{1}{1-x_i} = 1. \qedhere
         \end{equation*}
    \end{proof}

    \begin{proof}[Proof of Theorem \ref{thm:2}]
        We first recall that the map $f$ is defined over $K$ if and only if $\sigma(f)=f$ for all $\sigma \in \Gal(\overline{K}/K)$. (\cite[Exercise 1.12]{silverman-AEC}).

        Condition \eqref{cond_1} avoids the trivial obstruction of there not being enough distinct values for up to three distinct fixed points to be equal to their multipliers. So we focus on condition \eqref{cond_2}.

        Throughout the proof we may assume without loss of generality that the point at infinity is not a fixed point. We dehomogenize and denote the fixed points as $x_i$, $1 \leq i \leq \#\Fix(f)$ and the dehomogenized rational map as $F$.

        We first prove the conditions imply the existence of the desired form. Assume first that condition \eqref{cond_1} and \eqref{cond_2a} hold. Let $k = \min(\#\Fix(f),3)$. If $k = \#\Fix(f)$, then we may add arbitrary distinct rational points to the set of fixed points so that we have a Galois invariant set $\{z_1,z_2,z_3\}$ of distinct points that we wish to move via conjugation. The target set $\{t_1,t_2,t_3\}$ is the set of fixed point multipliers plus an arbitrary choice of distinct rational points. There is a unique $\rho\in\PGL_2(\overline{K})$ so that $f^{\rho}$ satisfies
        \begin{align*}
            F^{\rho}(\rho^{-1}(x_i)) &= \rho^{-1}(x_i) = F'(\rho^{-1}(x_i)), \quad 1 \leq i \leq \#\Fix(f)\\
            \rho^{-1}(z_i) &= t_i, \quad \#\Fix(f) < i \leq 3.
        \end{align*}
        We will see that $\rho \in \PGL_2(K)$ so that $f^{\rho}$ is defined over $K$. Write
        \begin{equation*}
            \rho = \begin{pmatrix}a&b\\c&d\end{pmatrix}.
        \end{equation*}
        Then we need to solve the system of equations
        \begin{equation*}
            at_i + b = z_i(ct_i+d), \quad 1 \leq i \leq 3.
        \end{equation*}
        Note that when $z_i$ is the fixed point $x_i$, then $t_i = F'(x_i)$.
        We know that (up to scaling) there is a unique solution $(a,b,c,d)$. Since the set of fixed points is Galois invariant and any additional points are rational, applying any element of the absolute Galois group to this system of equations leaves the system fixed and, thus, its solutions unchanged. Therefore, $(a,b,c,d)$ are defined over $K$ as $K$ is perfect.

        Assume now conditions \eqref{cond_1} and \eqref{cond_2b} are satisfied and condition \eqref{cond_2a} is not. Then $f$ has 4 distinct fixed points, of which there is no Galois invariant subset of size 3, and for which there are at least 3 distinct multipliers. Let $\alpha \in \PGL_2$ be an order $2$ automorphism of $f$. The automorphism $\alpha$ must exchange two fixed points with the same multiplier and fix the other fixed points (since they must have distinct multipliers in this case). After possible reordering the $x_i$, we know that $\alpha(x_i) = x_i$ for $i=1,2$ and $\alpha(x_3) = x_4$. Let $\rho \in \PGL_2(\overline{K})$ such that $\rho^{-1}(x_i) = F'(x_i)$ for $i=1,2,3$.
        Let $\sigma \in \Gal(\overline{K}/K)$. Then $\sigma$ fixes the points $\{\rho^{-1}(x_1),\rho^{-1}(x_2),\rho^{-1}(x_3),\rho^{-1}(x_4)\}$ or fixes $\{\rho^{-1}(x_1),\rho^{-1}(x_2)\}$ and exchanges $\rho^{-1}(x_3)$ and $\rho^{-1}(x_4)$. In particular, $f^{\rho}$ and $\sigma(f^{\rho})$ are degree $3$ rational functions with $4$ distinct fixed points with the same set of associated multipliers. So by (the corrected) Hutz-Tepper \cite[Theorem 6]{Hutz10} we have $\sigma(f^{\rho}) = f^{\rho}$. In particular, $f^{\rho}$ is defined over $K$.

        Assume now conditions \eqref{cond_1} and \eqref{cond_2c} are satisfied and conditions \eqref{cond_2a} and \eqref{cond_2b} are not. Then there are four distinct multipliers and we are in the case of Lemma \ref{lem_4_distinct}.

        For the other direction, without loss of generality assume that $f$ is in partial fixed-point multiplier form and defined over $K$. Since $f$ is defined over $K$, the fixed point equation (first dynatomic polynomial) is defined over $K$ and is Galois invariant, and thus $\Fix(f)$ is Galois invariant. If $f$ does not have a set of $\min(\#\Fix(f),3)$ distinct fixed points that are Galois invariant, then we must have $\#\Fix(f) > 3$, which implies $\#\Fix(f) = 4$ since $\#\Fix(f) \leq 4$ always. Let $\Fix(f) = \{x_1, x_2, x_3,x_4\}$. As $f$ is in partial fixed-point multiplier form, we can assume without loss of generality that $x_1$, $x_2$, and $x_3$ equal their multipliers. For the 4th fixed point $x_4$ there must be a $\sigma \in \Gal(\overline{K}/K)$ so that $\sigma(x_i) = x_4$ for some $i \in \{1,2,3\}$, say $\sigma(x_3) = x_4$. First note that if $f$ has a nontrivial automorphism $\alpha$ it must be order two. Recall that automorphisms map fixed points to fixed points, and further can only map fixed points to fixed points with the same multiplier. As each $x_i$ is distinct, $x_1$, $x_2$, and $x_3$ all have distinct multipliers, while the multiplier for $x_4$ must equal the multiplier for $x_3$. Any automorphism of $f$ must thus fix $x_1$ and $x_2$, while either fixing $x_3$ or mapping $x_3$ to $x_4$, which completely determines all automorphisms of $f$. If $f$ has a nontrivial automorphism $\alpha \in \PGL_2$, defined by $\alpha(x_1)=x_1$, $\alpha(x_3) = x_4$, and $\alpha(x_4) = x_3$, we are in condition \eqref{cond_2b}.
         Otherwise, we calculate
        \begin{align*}
            x_4 &= \sigma(x_3)\\
                &= \sigma(F'(x_3)) \quad \text{(because $f$ is in partial fixed-point multiplier form with}\\
                & \hspace*{90pt} \text{$x_i=F(x_i) = F'(x_i)$ for $i=1,2,3$)}\\
                &= F'(\sigma(x_3)) \quad \text{(because $F$ is defined over $K$)}\\
                &= F'(x_4)
        \end{align*}
        so that $f$ is, in fact, in fixed-point multiplier form (not just partial fixed point-multiplier form) and we are in condition \eqref{cond_2c}.
    \end{proof}
    We exhibit examples of each of the conditions (\ref{cond_2a}, \ref{cond_2b}, \ref{cond_2c}). Note that \eqref{cond_2b} is mutually exclusive to \eqref{cond_2c} since we cannot have an automorphism if the multipliers are all distinct.
    \begin{itemize}
        \item[\eqref{cond_2a}] We give examples with one, two, and three distinct fixed points.
        \begin{itemize}
            \item $F(z) = \frac{5z^3 - 1}{z^3 + z^2 + 6z - 4}$ has one fixed point, 1, and one multiplier, 1, and thus is in partial fixed-point multiplier form.
            \item $F(z) = \frac{z^3}{ -z^3 + 4z^2 - 3z + 1}$ with two distinct fixed points fixed points $\{0, 1\}$ and multipliers $\{0, 1\}$ is in partial fixed-point multiplier form.
            \item $F(z) = \frac{z^3}{-2z^3 + 9z^2 - 10z + 4}$ with three distinct fixed points $\{0,1,2\}$ and multipliers $\{0,1,2\}$ is in fixed-point multiplier form.
        \end{itemize}

        \item[\eqref{cond_2b}] $F(z) = \frac{18z^3}{-11z^3 + 57z^2 + 75z + 25}$ in partial fixed-point multiplier form with fixed points $\{-1,5,0,-5/11\}$ and respective multipliers $\{-1,5,0,5\}$ and a nontrivial automorphism group of order $2$ generated by $\begin{pmatrix}-1&0\\2&1 \end{pmatrix}$.
        \item[\eqref{cond_2c}] $F(z) = \frac{-3z^3 + z^2 - 2z - 2}{z^3 - 5z^2 + 4z - 4}$ with fixed points $\{-i,i,1-i,1+i\}$ in fixed-point multiplier form.
    \end{itemize}

\section{The First Dynatomic Field}\label{first_dynatomic}

   Let $L$ be the field of definition of the fixed points of $f$, called the \emph{first dynatomic field}. We call the Galois group $\Gal(L/K)$ the \emph{first dynatomic Galois group}. In what follows, $S_n$ denotes the symmetric group on $n$ letters, while $C_n$ denotes the cyclic group on $n$ letters. In the proof of Theorem~\ref{thm:2}, the main idea is to use the existence of a Galois invariant subset that contains three fixed points, so it is no surprise that if $f$ has a conjugate in partial fixed-point multiplier form, then we gain some control on the first dynatomic Galois group of $f$.

   We first state a lemma which is essentially an explicit version of Theorem \ref{thm:2}(\ref{cond_2a}). In particular, finding a $\rho$ so that $f^{\rho}$ is in partial fixed-point multiplier form is equivalent to choosing three fixed points with distinct multipliers.
   \begin{lem}\label{lem:possible mobius transforms}
       Let $f$ be a degree 3 rational map defined over a perfect field $K$ with $\# \text{Fix}(f) \geq 3$. Let $\Fix(f) = \{x_1, x_2, x_3, x_4\}$ where possibly $x_3 = x_4$, and let $t_i = f'(x_i)$ for $1 \leq i \leq 4$.

       An element $\rho \in \PGL_2(\overline{K})$ such that $f^\rho$ is in partial fixed-point multiplier form is determined by choosing a subset of $3$ distinct points $\{x_{i_1}, x_{i_2}, x_{i_3}\}$ of $\text{Fix}(f)$ with distinct multipliers. Furthermore, if the set $\{x_{i_1}, x_{i_2}, x_{i_3}\}$ is invariant under the first dynatomic Galois group, then $\rho$ is defined over $K$.
   \end{lem}
   \begin{proof}
    We have that $\Fix(f^\rho) = \{\rho^{-1}(x_i) : i=1,2,3,4\}$. Recall that the multipliers (as a set) are invariant under conjugation, and thus the set of multipliers of $f^\rho$ is $\{t_1,t_2,t_3,t_4\}$. As $f^\rho$ is in partial fixed-point multiplier form, there is some subset $\{\rho^{-1}(x_{i_1}), \rho^{-1}(x_{i_2}), \rho^{-1}(x_{i_3})\}$ such that
   \begin{align*}
       \rho^{-1}(x_{i_l}) = t_{j_l}, \quad 1 \leq l \leq 3
   \end{align*}
   for some subset $\{t_{j_1}, t_{j_2}, t_{j_3}\}$ of the multipliers. The existence of a solution (or rational solution) to this system is demonstrated in the proof of Theorem \ref{thm:2}(\ref{cond_2a}).
   \end{proof}
   Note that since an element of $\PGL_2$ is uniquely determined by what it does to three distinct points, the $\PGL_2$ elements obtained from two different subsets in Lemma \ref{lem:possible mobius transforms} can only be the same when $f$ has four distinct fixed points and $f^{\rho}$ is in fixed-point multiplier form.

   We are also able to count how many different conjugates of $f$ are in partial fixed-point multiplier form and defined over $K$, as in Corollary \ref{cor:1}.
    \begin{cor}\label{cor:1}
    Let $f$ be a rational map of degree 3 defined over a perfect field $K$. Assume $\Aut(f)$ is trivial. Suppose that $\rho\in \PGL_2(\overline{K})$ is such that $f^{\rho}$ is in partial fixed-point multiplier form and defined over $K$. Then, $\rho$ is defined over $K$. Further, we can count the number of such $\rho$:
    \begin{enumerate}
        \item $\rho$ is unique if and only if $f^{\rho}$ is in fixed-point multiplier form or the first dynatomic Galois group is isomorphic to $S_3$ or $C_3$.
        \item There are exactly two distinct $\rho\in\PGL_2(K)$ such that $f^{\rho}$ is in partial fixed-point multiplier form if and only if $f$ has four fixed points, and one of the following is true:
        \begin{enumerate}
            \item the first dynatomic Galois group is $C_2$; \label{cor_1a}
            \item $f$ has two distinct fixed points with the same multiplier. \label{cor_1b}
        \end{enumerate}
        \item There are exactly four distinct $\rho\in\PGL_2(K)$ such that $f^{\rho}$ is in partial fixed-point multiplier form if and only if the first dynatomic Galois group is trivial and $f$ has four fixed points.
    \end{enumerate}
    \end{cor}
    \begin{proof}
    Note first that as $\Aut(f)$ is trivial, we are not in case (2b) of Theorem \ref{thm:2}. We show that $\rho$ is defined over $K$. If we are in case (2a) of Theorem \ref{thm:2}, then we have a $\Gal(\overline{K}/K)$ subset of fixed points and Lemma \ref{lem:possible mobius transforms} or the proof of case (2a) shows that $\rho$ is defined over $K$. In case (2c) of Theorem \ref{thm:2}, letting $\{x_i\}$ denote the set of fixed point and $\{t_i\}$ denote the set of multipliers, we have that $\rho$ is a solution to the set of equations
    \begin{align*}
        \rho^{-1}(x_i) = t_i, \quad 1 \leq i \leq 4.
    \end{align*}
    As the set of fixed points is Galois invariant, the above set of equations is Galois invariant so that $\rho$ is defined over $K$.
    \hfill
    \begin{enumerate}
        \item Suppose that there is a unique $\rho\in\PGL_2(\overline{K})$ such that $f^{\rho}$ is in partial fixed-point multiplier form defined over $K$. If $f^{\rho}$ is in fixed-point multiplier form, then we are done. Otherwise, we are in case (2a) of Theorem~\ref{thm:2}. In the proof of case (2a), we add an arbitrary rational point to find $\rho$ if $f$ has less than three fixed points. Thus, in case (2a), if $f$ has less than three fixed points, $\rho$ is not unique. There must therefore exist a Galois invariant set that contains three fixed points of $f$. This Galois invariant set is unique since, by hypothesis, there is a unique $\rho$, and each $\rho$ arises from a Galois invariant set of three fixed points by Lemma \ref{lem:possible mobius transforms}. If $f$ has only three fixed points, then $f^{\rho}$ is in fixed-point multiplier form. Thus, we assume $f$ has four fixed points for the following argument. Obviously, the fixed point not in the Galois invariant set must be defined over $K$. We have thus shown that the first dynatomic field is the splitting field of a cubic, and hence the order of the Galois group must divide six. We now show that the first dynatomic Galois group is not trivial and not $C_2$. If the first dynatomic Galois group is trivial, then every fixed point is rational, which violates the existence of a unique Galois invariant set of size three. Similarly, the first dynatomic Galois group cannot be isomorphic to $C_2$; otherwise, it violates the unique
        existence of a Galois invariant set of three fixed points. Conversely, if $f$ has a conjugate in fixed-point multiplier form, then the conjugate is obviously unique. If the first dynatomic Galois group is isomorphic to $C_3$ or $S_3$, then we can quickly see that $\rho$ is unique as there is a unique set of three fixed points which is Galois invariant.
        \item We begin with the converse direction. Assume first that $f$ has four fixed points, and the first dynatomic Galois group is $C_2$. To construct a $\rho$ such that $f^\rho$ is in partial fixed-point multiplier form, by Lemma \ref{lem:possible mobius transforms} it is sufficient to choose a Galois invariant set of three fixed points which all have distinct multipliers. Let $\sigma$ be the generator of the first dynatomic Galois group, and let $x_1$ and $x_2$ be fixed points such that $\sigma(x_1) = x_2$. Then, there are two possible Galois invariant sets of fixed points of size three, namely $\{x_1, x_2, x_3\}$ and $\{x_1, x_2, x_4\}$, and hence exactly two choices for $\rho$. Similar logic applies when $f$ has two distinct fixed points with the same multiplier: to construct a set of three fixed points with all distinct multipliers, we must pick one of the fixed points with the same multiplier, giving two choices for such a set, and hence two choices for $\rho$.

        Now we show that the existence of two such $\rho$ implies the conditions. Since we assume $\Aut(f)$ is trivial, Theorem~\ref{thm:2} implies there must exist a Galois invariant set of at least three fixed points. Since there are two choices of $\rho$, there are two Galois invariant sets of three fixed points. If condition \eqref{cor_1b} holds, i.e. $f$ has two distinct fixed points with the same multiplier, then we are done. Otherwise, there are four distinct multipliers. If the dynatomic Galois group were trivial, then there would be four Galois invariant sets of three fixed points, and, hence, four distinct such M\"{o}bius transformations $\rho$ (Lemma \ref{lem:possible mobius transforms}). Since there are only two such $\rho$ by assumption, the dynamical Galois group cannot be trivial. The only remaining possibility is that the dynamical Galois group is order two and, hence, isomorphic to $C_2$.

        \item This is the remaining case. If $f$ has at most three distinct fixed points, then $f^{\rho}$ would be in fixed-point multiplier form, which is covered in part (1). Thus, we can assume $f$ has four distinct fixed points. Each $\rho$ corresponds to a Galois invariant subset of three fixed points of $f$ by Lemma \ref{lem:possible mobius transforms}. The cases of one distinct and two distinct $\rho$ are covered in (1) and (2), since there are at most four size three subsets of a four element set,  we need to show there cannot be three distinct $\rho$. If there are 3 distinct $\rho$, then there are 3 such distinct subsets and all fixed points are rational. But if all fixed points are rational, we have four distinct $\rho$s by choosing arbitrary 3 fixed points among 4. Finally, none of $f^\rho$ are in fixed-point multiplier form or otherwise we are in (1). \qedhere
    \end{enumerate}
    \end{proof}

    \begin{cor}\label{cor:2}
     Let $f$ be a degree 3 rational map defined over a perfect field $K$, and let $L=K(f(x)-x)$ be the first dynatomic field of $f$. Assume $\Aut(f)$ is of order $2$. Then, there exists $\rho\in \PGL_2(\overline{K})$ such that $f^{\rho}$ is in partial fixed-point multiplier form if and only if both of the following hold:
     \begin{enumerate}
       \item The first dynatomic group $\Gal(L/K)$ is isomorphic to the trivial group, $C_2$, or $C_2\times C_2$.
       \item $f$ has four fixed points.
     \end{enumerate}
     Furthermore, there always exist exactly two distinct such $\rho$.
    \end{cor}
    \begin{proof}
    Assume there is a $\rho \in \PGL_2(\overline{K})$ so that $f^{\rho}$ is in partial fixed-point multiplier form. Since $\Aut(f)$ is non-trivial, there are two distinct fixed points with the same multiplier. In particular, $f$ cannot be in fixed-point multiplier form. Since $f^{\rho}$ is in partial fixed-point multiplier form, we must then have 4 distinct fixed points. Finally, there are exactly two distinct sets of three fixed points and $f^{\rho}$ is not in fixed-point multiplier form, so by Lemma \ref{lem:possible mobius transforms} there are exactly two M\"{o}bius transformations $\rho$ conjugating $f$ to a partial fixed-point multiplier form.

    Now assume the two conditions hold. In particular, $f$ has four distinct fixed points and the first dynatomic Galois group is a subgroup of $C_2 \times C_2$. Label the four fixed points as $x_1$, $x_2$, $x_3$, and $x_4$. Then, after possibly relabeling, any element of the Galois group swaps $x_1$ and $x_2$ or $x_3$ and $x_4$ and the nontrivial automorphism swaps $x_3$ and $x_4$. In particular, there are exactly two sets of three distinct fixed points with distinct multipliers: $\{x_1,x_2,x_3\}$ and $\{x_1,x_2,x_4\}$. By Lemma \ref{lem:possible mobius transforms} these sets corresponds to distinct $\rho$ for which $f^{\rho}$ is in partial fixed-point multiplier form. Note that these $\rho$ are distinct since $f$ cannot be in fixed-point multiplier form with four distinct fixed points and only three distinct multiplier values.
    \end{proof}
    Corollary~\ref{cor:3} provides a convenient method to determine whether a conjugate of a rational map for partial fixed-multiplier form exists.
    \begin{cor}\label{cor:3}
        Let $f$ be a degree three rational map written as $\frac{p(x)}{q(x)}$ for some polynomials $p$ and $q$ defined over a perfect field $K$. We further suppose that $f$ has at least three distinct fixed point multipliers. Then, there exists some $\rho\in\PGL_2(\overline{K})$ such that $f^{\rho}$ is in partial fixed-point multiplier form if and only if at least one of the following is true.
        \begin{enumerate}
            \item The fixed points for some conjugate of $f$ satisfies the hypotheses in Lemma~\ref{lem_4_distinct}.\label{cor_1}
            \item The degree of $q(x)$ is equal to or less than $2$.\label{cor_2}
            \item $p(x)-xq(x)$ is of degree $4$ and has at least one zero in $K$ with $\text{char}(K)\neq 2$.\label{cor_3}
            \item $f$ has a unique nontrivial automorphism.\label{cor_4}
        \end{enumerate}
    \end{cor}
    \begin{proof}
    Assume there exists $\rho\in\PGL_2(\overline{K})$ such that $f^{\rho}$ is in partial fixed-point multiplier form. If $q(x)$ has degree equal to or less than $2$, then we are in case \eqref{cor_2}.

    Now, we can assume $q$ has degree $3$. If $f^\rho$ is in fixed-point multiplier form, then we are in the case of Lemma \ref{lem_4_distinct}. If $f^\rho$ is not in fixed-point multiplier form, then we must either have $\#\Aut(f)=2$ or $f$ has a Galois invariant set containing three fixed points (Theorem \ref{thm:2}). Note that all the fixed points of $f$ are finite, so they are the solutions of
    \begin{equation} \label{eq2}
        p(x)-xq(x)=0.
    \end{equation}
    If there are three roots forming a Galois invariant set, then one of the roots of \eqref{eq2} must be in $K$, or there are only three solutions to \eqref{eq2}. If there are only three roots, then one of the roots is a double root. The double root may be in $K$ if the root is ramified over $K$. Since \eqref{eq2} is a degree four polynomial defined over $K$ and $\text{char}(K)\neq 2$, \eqref{eq2} will not have an irreducible ramified factor. Therefore, the double root must be in $K$.

    Conversely, we demonstrate that each case implies $\rho\in\PGL_2(\overline{K})$ such that $f^{\rho}$ is in partial fixed-point multiplier form.

    Case (1) assumes the hypotheses of Lemma~\ref{lem_4_distinct} (four distinct fixed point multipliers which form a Galois invariant set) and Lemma~\ref{lem_4_distinct} conclude the existence of a unique $\rho$ so that $f^{\rho}$ is in fixed-point multiplier form.

    Cases (2) and (3) imply there is a Galois invariant set of three fixed points of $f$. We can then apply Lemma \ref{lem:possible mobius transforms} to conclude the existence of a $\rho$ so that $f^{\rho}$ is in partial-fixed point multiplier form.

    Finally, case (4) together with the assumption of at least three distinct fixed points puts us in the case of Theorem \ref{thm:2}(2b). Theorem \ref{thm:2} then concludes the existence of a $\rho$ so that $f^{\rho}$ is in partial-fixed point multiplier form.
    \end{proof}
    To find a rational map that has no conjugate defined over the field of moduli with the property that the fixed points are the multipliers, we simply take one that does not satisfy the assumptions of Theorem \ref{thm:2}.
    \begin{ex}
        Let $f(x) = \frac{x^3+x+1}{x^3}$. Then $f$ does not have a Galois invariant set of three fixed points, a nontrivial automorphism, nor is in fixed-point multiplier form. The fixed points of $f$ are
        \begin{equation*}
            \left\{i, -i, \frac{1+\sqrt{5}}{2}, \frac{1-\sqrt{5}}{2} \right\}
        \end{equation*}
        with associated multipliers
        \begin{equation*}
            \left\{-2i-3, 2i+3, \frac{5\sqrt{5}-13}{2}, \frac{-5\sqrt{5}-13}{2} \right\}
        \end{equation*}
        That $f$ has no conjugate in fixed-point multiplier form can then be easily checked by computer by examining all possible conjugations moving fixed points to multipliers.
    \end{ex}
    \begin{rem}
        Note that the set of maps which cannot be conjugated to partial fixed-point multiplier form defined over $K$ is Zariski dense in the space of degree $3$ rational maps.
        This is because conditions (\ref{cor_1}, \ref{cor_2}, and \ref{cor_4}) from Corollary \ref{cor:3} are closed conditions and condition \eqref{cor_3} is contained in a (type II) thin set in the sense of Serre.
    \end{rem}

\bibliographystyle{plain}

\end{document}